\documentclass{amsart}

\usepackage{amsmath,amsthm,amsfonts}

\usepackage{tikz}
\usetikzlibrary{calc,fadings,decorations.pathreplacing}
\usepackage{verbatim}

\usepackage{hyperref}
\usepackage{enumerate}
\usepackage{graphicx}
\usepackage{multirow}
\usepackage{array}
\usepackage{booktabs}

\newcolumntype{C}[1]{>{\centering\let\newline\\\arraybackslash\hspace{0pt}}m{#1}}

\theoremstyle{plain}
\newtheorem{theorem}{Theorem}[section]

\newtheorem{proposition}[theorem]{Proposition}
\newtheorem{lemma}[theorem]{Lemma}
\newtheorem{conjecture}[theorem]{Conjecture}
\newtheorem{observation}[theorem]{Observation}

\theoremstyle{definition}
\newtheorem{definition}[theorem]{Definition}
\newtheorem{question}[theorem]{Question}

\numberwithin{equation}{section}
\numberwithin{table}{section}
\numberwithin{figure}{section}

\newcommand{\msr}[1]{{\rm msr}(#1)}
\newcommand{\zplus}[1]{Z_+(#1)}

\DeclareMathOperator{\rank}{rank}

\title
[
Orthogonal representations of Steiner triple system incidence graphs
]
{
Orthogonal representations of Steiner triple system \\ incidence graphs
}

\author{Louis Deaett}
\address{Department of Mathematics, Quinnipiac University, Hamden, CT 06518, USA.}

\author{H.\ Tracy Hall}
\address{Department of Mathematics, Brigham Young University, Provo, UT 84602, USA}

\keywords{faithful orthogonal representation; Heawood graph; Steiner triple system; minimum rank problem; minimum semidefinite rank}

\subjclass[2010]{Primary: 52C99; Secondary: 05C50}

\begin{document}
\maketitle

\begin{abstract}
The unique Steiner triple system of order $7$ has a point-block incidence graph known as the Heawood graph.  Motivated by questions in combinatorial matrix theory, we consider the problem of constructing a faithful orthogonal representation of this graph, i.e., an assignment of a vector in $\mathbb C^d$ to each vertex such that two vertices are adjacent precisely when assigned nonorthogonal vectors. We show that $d=10$ is the smallest number of dimensions in which such a representation exists, a value known as the \textit{minimum semidefinite rank} of the graph, and give such a representation in $10$ real dimensions. We then show how the same approach gives a lower bound on this parameter for the incidence graph of any Steiner triple system, and highlight some questions concerning the general upper bound.
\end{abstract}

\begin{center}\today\end{center}

\section{Introduction}
\label{sec:intro}

Fundamental to what follows is the idea of assigning a vector to each vertex of a graph so that the inner products among the vectors in some way reflect the adjacency relation on the vertices.  A geometric representation of this sort may then be useful in studying properties of the graph.  This approach dates back at least to the celebrated work of Lov\'asz \cite{lovasz} in determining the Shannon capacity of the $5$-cycle;  see also \cite{parsons} for a unifying discussion.  The following definition provides one realization of this idea.

\begin{definition}\label{def:orth_rep}
Let $G$ be a graph and $X$ be an inner product space.  An \textit{orthogonal representation of $G$ in $X$} is a function $r:V(G) \rightarrow X$ such that two vertices of $G$ are adjacent if and only if they are mapped by $r$ to nonorthogonal vectors, i.e., for any distinct $u,v \in V(G)$,
\begin{equation}\label{eqn:orth_rep_condition}
\langle r(u), r(v) \rangle \neq 0
\, \Longleftrightarrow \,
\{u,v\} \in E(G).
\end{equation}
\end{definition}

We note that the notion requiring only the forward direction of \eqref{eqn:orth_rep_condition} has received a good deal of attention; many authors would refer to the notion set out in Definition \ref{def:orth_rep} as that of a \textit{faithful} orthogonal representation. Also, what some authors would consider an orthogonal representation of $G$ would be considered by others to be an orthogonal representation of the complement of $G$. Variations of this sort must be kept in mind when considering the related literature; some results derived in the context of different such choices are surveyed in each of \cite{haynes} and \cite{parsons}.

The particular notion of an orthogonal representation given by Definition \ref{def:orth_rep} has relevance to combinatorial matrix theory in the context of certain variants of the \emph{minimum rank problem}, which, broadly construed, calls for finding the smallest possible rank among all matrices meeting a given combinatorial description.  Instances of this problem arise naturally in applications such as computational complexity theory \cite{deaett_srinivasan} and quantum information theory \cite{scarpa_severini}.  Often, a matrix is first required to be symmetric (or Hermitian) and then further conditions are imposed in terms of the graph whose edges correspond to the locations of the off-diagonal nonzero entries of the matrix. This is made precise by the following definition.  Note that, in all that follows, we denote by $A_{ij}$ the entry in row $i$ and column $j$ of matrix $A$.

\begin{definition}
Let $A$ be an $n\times n$ Hermitian matrix.  The \textit{graph of $A$} is the unique simple graph on vertices $v_1,\ldots,v_n$ such that, for every $i\not= j$, vertices $v_i$ and $v_j$ are adjacent if and only if $A_{ij} \neq 0$.
\end{definition}

The associated minimum rank problem is to determine the smallest rank among the Hermitian (or real symmetric) matrices with a fixed graph,
a value known as the \textit{minimum rank} of the graph. This problem has received considerable attention in combinatorial matrix theory; see \cite{handbook_min_rank_chapter} for a survey.  The present work bears on a variant of the problem in which only positive semidefinite Hermitian matrices are considered.  In particular, we study the graph invariant defined as follows.

\begin{definition}\label{def:msr}
Let $G$ be a simple graph on $n$ vertices.  The \textit{minimum semidefinite rank} of $G$ is the smallest rank among all positive semidefinite Hermitian matrices with graph $G$.  This value is denoted by $\msr{G}$.
\end{definition}

The following observation connects this variant of the minimum rank problem with the notion of an orthogonal representation.  It
 is a simple consequence of the characterization of positive semidefinite matrices as Gram matrices.

\begin{observation}
The smallest $d$ such that $G$ has an orthogonal representation in $\mathbb C^d$ is $d=\msr{G}$.
\end{observation}

The smallest $d$ allowing an orthogonal representation of $G$ in $\mathbb R^d$ may also be of interest, in which case the minimum of Definition \ref{def:msr} can be taken over the real symmetric matrices; we refer to this value as the \textit{minimum semidefinite rank of $G$ over $\mathbb R$}.

The ordinary Laplacian matrix of a graph $G$ on $n$ vertices shows $\msr{G}$ to be well-defined and at most $n-1$.  In addition, the minimum semidefinite rank is additive on the connected components of a graph, so that it is sufficient to consider connected graphs only.  The question of how combinatorial properties of a graph relate to its minimum semidefinite rank has received a good deal of interest; see, e.g., \cite{booth_et_al_2011, booth_et_al_2008} and \cite[Section 46.3]{handbook_min_rank_chapter}.  One simple result is the following.

\begin{theorem}[\cite{van_der_holst}]\label{thm:random_msr_facts}
	If $G$ is a cycle on $n$ vertices, then $\msr{G} = n-2$.
\end{theorem}

The motivation for the present work begins with a result of \cite{rosenfeld} that can be recast as follows.

\begin{theorem}[\cite{rosenfeld}]\label{thm:rosenfeld}
If $G$ is a connected triangle-free graph on $n \ge 2$ vertices, then $\msr{G} \ge \frac 12 n$.
\end{theorem}

The question of how Theorem \ref{thm:rosenfeld} might generalize has received some attention.  For instance, the implications of replacing the triangle-free condition with a larger upper bound on the clique number are explored in \cite{furedi}.  In \cite{deaett_msr}, it was observed that (except in trivial cases) a graph meeting the lower bound of Theorem \ref{thm:rosenfeld} must have a girth of $4$, suggesting that, in seeking a generalization, the condition that the graph be triangle-free be viewed as a lower bound on its girth.  In particular, the following conjecture was put forward.

\begin{conjecture}[\cite{deaett_msr}]\label{conj:girth_conjecture}
Suppose $G$ is a connected graph on $n \ge 2$ vertices and $k$ is an integer with $k \ge 4$.  If $G$ has girth at least $k$, then
$\msr{G} \ge \left(\frac{k-2}k\right)\!n$.
\end{conjecture}

In light of Theorem \ref{thm:random_msr_facts}, this conjecture may be viewed as asserting that, among the connected graphs of girth at least $k$, the minimum semidefinite rank as a fraction of the number of vertices is minimized by the $k$-cycle.  While Conjecture \ref{conj:girth_conjecture} was found to hold for all graphs on at most $7$ vertices, it was suggested that a revealing test case might be provided by the \textit{cage graphs}, defined as follows.

\begin{definition}
	A graph that has girth $g$ in which each vertex has degree $d$ is called a \textit{$(d,g)$-cage} when no graph on fewer vertices has both of those properties.
\end{definition}

The well-known Petersen graph, with $10$ vertices, is the unique $(3,5)$-cage.  Its minimum semidefinite rank is $6$, meeting the lower bound of Conjecture \ref{conj:girth_conjecture}.  There is also a unique $(3,6)$-cage, known as the Heawood graph, with $14$ vertices, shown in Figure \ref{fig:heawood_graph}.   Conjecture \ref{conj:girth_conjecture} would require its minimum semidefinite rank to be at least $10$.

\begin{figure}[h]
\begin{tikzpicture}

\def \r {1.75}
\def \n {7}
\def \pt_mark_radius {1.75}

\foreach \s in {0,...,6}
{
	\coordinate[coordinate] (A\s) at ({360/\n * \s + 80} : \r);
    \coordinate[coordinate] (B\s) at ({360/\n * \s + 100} : \r);
	\path[draw, fill] (A\s) circle[radius=\pt_mark_radius pt];
	\path[draw, fill] (B\s) circle[radius=\pt_mark_radius pt];
}

\foreach \t in {0,...,6}
{
	\draw[thick] (A\t) -- (B\t);
	\draw[thick] let \n1={int(mod(\t+2,7))} in (A\t) -- (B\n1);
	\draw[thick] let \n1={int(mod(\t+6,7))} in (A\t) -- (B\n1);
}

\end{tikzpicture}
\caption{The Heawood graph.\label{fig:heawood_graph}}
\end{figure}
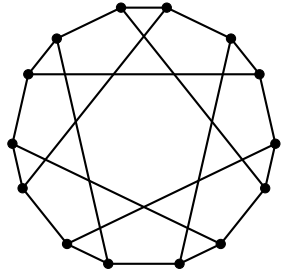

Unfortunately, the problem of determining for a given graph the value of its minimum rank or minimum semidefinite rank may be very difficult. One of the few general techniques (introduced in \cite{original_ZF_paper}) that is available for lower-bounding the minimum rank involves computing a \textit{zero forcing parameter} for the graph. Such a parameter gives an upper bound on the dimension of the null space, and hence a lower bound on the rank, of a matrix by exploiting how the graph of the matrix constrains the zero-nonzero patterns (i.e., supports) that may occur among its null vectors.

A variant of the zero forcing technique specific to the positive semidefinite case was introduced in \cite{barioli_et_al}.  Namely, the \textit{positive semidefinite zero forcing number} of $G$, denoted $\zplus{G}$, is defined for any graph $G$.   While a precise definition of $\zplus{G}$ is beyond the scope of this paper, we note that the definition is purely combinatorial, so that $\zplus{G}$ may be considered from a strictly graph-theoretic perspective.  (See, e.g., \cite{psd_zf_paper}.)  Nevertheless, $\msr{G} \ge n - \zplus{G}$ for every graph $G$, though the gap may be arbitrarily large \cite{mitchell_et_al}.

With $H$ denoting the Heawood graph, computer calculation using \cite{sage_min_rank_library} gives $Z_+(H) = 5$, implying that $\msr{H} \ge 9$. The results developed in this paper show that in fact $\msr{H}=10$. This is accomplished via a geometric approach that may be applied to the graph describing the incidence structure of any Steiner triple system.  The Heawood graph is one such graph; we treat the general case in Section \ref{sec:generalization}.

The remainder of this paper is organized as follows.  Section \ref{sec:heawood} presents the necessary background regarding the Heawood graph and its relevant connections with other mathematical objects.  Section \ref{sec:heawood_msr} develops the main results of the paper, establishing upper and lower bounds on the minimum semidefinite rank of the Heawood graph.  Section \ref{sec:generalization} explores how the same approach may be applied to the incidence graph of any Steiner triple system, and exhibits further bounds derived by this method.  Finally, Section \ref{sec:conclusion} highlights some questions suggested by this work and possible directions for future research.

\section{The Heawood graph}\label{sec:heawood}

The Heawood graph, shown in Figure \ref{fig:heawood_graph}, has served as an important example in the study of minimum rank problems.  For example, the complement of this graph was used in \cite{barioli_et_al_JGT} to give separation between various minimum rank parameters and corresponding combinatorial bounds.  Also, \cite{canto} presented the first example of a zero-nonzero pattern for which the minimum rank (over the reals) was unequal to a combinatorial lower bound known as the \textit{triangle number}, and the pattern given was exactly that of the biadjacency matrix (defined in Section \ref{sec:heawood_msr}) of the Heawood graph.

The minimum rank of the Heawood graph may be obtained as follows.  First, the ordinary zero forcing number (not the positive semidefinite variant) gives a lower bound of $8$.  Meanwhile, the adjacency matrix $A$ of the graph has the eigenvalue $\sqrt 2$ with multiplicity $6$, so that $\rank(A - \sqrt 2 I)=14-6=8$ gives a corresponding upper bound.  Hence, the minimum rank of the Heawood graph is $8$.

In the context of minimum semidefinite rank, interest in the Heawood graph emerged due to properties making it an attractive test case for Conjecture \ref{conj:girth_conjecture}, as outlined in Section \ref{sec:intro}.  For what follows, the most important way to view the Heawood graph is through its connection with the \textit{Fano plane}, the finite projective plane of order $2$,  illustrated in Figure \ref{fig:fano_plane}.  This is a finite geometry comprising seven points and seven lines in which each line contains exactly three points and each point lies on exactly three lines.  Hence, its points and lines give a Steiner triple system (in fact, the unique one) of order $7$.  We return to this connection in Section \ref{sec:generalization}; for now, we need to note only the following.

\begin{figure}[h]
\begin{tikzpicture}

\def \r {0.8}
\def \pt_mark_radius {1.75}

\coordinate[coordinate] (O) at (0,0);
\path[draw, fill] (O) circle[radius=\pt_mark_radius pt];

\draw[thick] (O) circle (\r);

	\foreach \s in {1,2,3}
	{
		\coordinate[coordinate] (A\s) at ({120 * \s - 90} : \r);
		\coordinate[coordinate] (B\s) at ($(O)!{-2*\r cm}!(A\s)$);
		\draw[thick] (A\s) -- (B\s);
		
		\path[draw, fill] (A\s) circle[radius=\pt_mark_radius pt];
		\path[draw, fill] (B\s) circle[radius=\pt_mark_radius pt];
	}
	
	\draw[thick] (B1) -- (B2) -- (B3) -- (B1);

	\draw (B3) node[above=0.5ex]				{$\sf 2$};
	\draw (B2) node[below=1.5ex,right=0.0ex]	{$\sf 4$};
	\draw (B1) node[below=1.5ex,left=0.0ex]		{$\sf 7$};
	\draw (A1) node[above=0.5ex,right=0.25ex]	{$\sf 5$};
	\draw (A3) node[below=0.25ex]				{$\sf 6$};
	\draw (A2) node[above=0.5ex,left=0.25ex]	{$\sf 3$};
	\draw (O) node[above=1.9ex,right=-0.6ex]	{$\sf 1$};
	
\end{tikzpicture}\caption{The Fano plane.}\label{fig:fano_plane}
\end{figure}
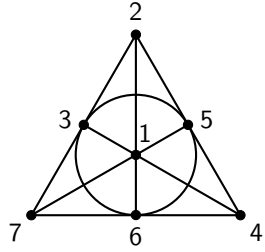

\begin{observation}\label{obs:lines_of_fano_plane}
The points of the Fano plane may be identified with the integers $1,2,\ldots,7$ so that the set of its lines becomes
\begin{equation}\label{eqn:fano_plane_lines}
\{\{1,2,6\}, \{2,3,7\}, \{1,3,4\},\{2,4,5\},\{3,5,6\},\{4,6,7\}, \{1,5,7\}\}.
\end{equation}
Figure \ref{fig:fano_plane} shows the points of the Fano plane labeled to reflect such an identification.
\end{observation}

The Heawood graph is the \textit{point-edge incidence graph} of the Fano plane.  That is, its vertices can be partitioned into two independent sets, one in correspondence with the points of the Fano plane, and the other in correspondence with its lines, such that a point and a line are incident precisely when the corresponding vertices are adjacent.  Through this connection, many of the properties of the Heawood graph that we will need follow from properties of the Fano plane.

One such property concerns the smallest size of a set from which one may color the points of the Fano plane without inducing a \textit{monochromatic line}, a line with all of its points colored the same.  In particular, the following simple fact (a special case of a result of \cite{rosa}; see Section \ref{sec:generalization}) is straightforward to verify.

\begin{lemma}\label{lem:fano_plane_not_2_colorable}
Every $2$-coloring of the points of the Fano plane induces a monochromatic line.
\end{lemma}

\section{The minimum semidefinite rank of the Heawood graph}\label{sec:heawood_msr}

The goal of this section is to establish that the minimum semidefinite rank of the Heawood graph is 10, and that this in fact holds over $\mathbb R$ as well.  We begin by noting that if a graph is bipartite, then a special attack is possible on the problem of determining its minimum semidefinite rank. A brief argument in the case of the Heawood graph follows; for a general discussion, see \cite[Theorem 5.3]{deaett_msr} or \cite[Proposition 3.1]{jiang_et_al}.

\begin{lemma}\label{lm:condition_on_terminal_Us}
Let $F$ be $\mathbb R$ or $\mathbb C$.  The Heawood graph has an orthogonal representation in $F^{7+n}$ if and only if some matrix $A \in M_{7+n,7}(F)$ with mutually orthogonal columns has the form 
\begin{equation}\label{eq:heawood_matrix_with_terminal_u_vectors}
{\left[ \begin{array}{lllllllll}
  * & * & 0 & 0 & 0 & * & 0 \\
  0 & * & * & 0 & 0 & 0 & * \\
  * & 0 & * & * & 0 & 0 & 0 \\
  0 & * & 0 & * & * & 0 & 0 \\
  0 & 0 & * &  0& * & * & 0 \\
  0 & 0 & 0 & * & 0 & * & * \\
  * & 0 & 0 & 0 & * & 0 & * \\ \hline
  u_1 & u_2 & u_3 & u_4 & u_5 & u_6 & u_7
\end{array} \right]},
\end{equation}
where each $\ast$ denotes a nonzero entry, and each $u_i$ is a (column) vector in $F^n$.
\end{lemma}

\begin{proof}
Let $B$ be a \textit{biadjacency matrix} for the Heawood graph; that is, $B$ is a $(0,1)$-matrix with rows in correspondence with one of its partite sets and columns in correspondence with the other such that $B_{ij}=1$precisely when the vertices corresponding to row $i$ and column $j$ are adjacent.  Then, subject to an appropriate ordering of its rows and columns, the zero-nonzero pattern of $B$ is given by the upper $7\times 7$ submatrix of \eqref{eq:heawood_matrix_with_terminal_u_vectors}.  That is, $A \in M_{7+n,7}(F)$ is of the form \eqref{eq:heawood_matrix_with_terminal_u_vectors} if and only if the upper $7\times 7$ submatrix of $A$ has the zero-nonzero pattern of $B$.  Thus, if such a matrix $A$ exists with mutually orthogonal columns, then the columns of $A$ together with the initial $7$ unit coordinate vectors in $F^{7+n}$ form an orthogonal representation of the Heawood graph.

Conversely, given an orthogonal representation of the Heawood graph in $F^{7+n}$, it may be assumed (subject to an appropriate unitary transformation) that one of the partite sets is assigned the first $7$ standard coordinate vectors.
Taking the vectors assigned to the other partite set as the columns of a matrix $A$ then gives $A \in M_{7+n,7}(F)$ of the form \eqref{eq:heawood_matrix_with_terminal_u_vectors} with mutually orthogonal columns.
\end{proof}

Hence, the minimum semidefinite rank of the Heawood graph is seen to be the smallest value of $7+n$ such that some $A \in M_{7+n,7}(\mathbb C)$ of the form \eqref{eq:heawood_matrix_with_terminal_u_vectors} has mutually orthogonal columns.  Lemma \ref{lm:condition_on_vectors} gives a useful reformulation of this condition; its proof relies on the following  observation.

\begin{observation}\label{obs:combinatorics_of_fano_plane_upper_7_by_7}
Let $A$ be a matrix of the form \eqref{eq:heawood_matrix_with_terminal_u_vectors}.  In particular, each of the seven lines of the Fano plane, as given in \eqref{eqn:fano_plane_lines}, gives the locations of the nonzero entries within one of the first seven rows of $A$.
Since each pair of points of the Fano plane lies on exactly one line, it follows that,
for every pair of distinct columns $i$ and $j$ of $A$, there exists a unique $k \in \{1,2,\ldots,7\}$ such that both columns have a nonzero entry in row $k$.  Hence, when $A$ has entries from $\mathbb C$, the two columns are orthogonal if and only if
$A_{ki}\overline{A_{kj}} = -\langle u_i, u_j \rangle$.
\end{observation}

Note that, combinatorially, Observation \ref{obs:combinatorics_of_fano_plane_upper_7_by_7} derives from the fact that the Heawood graph is the incidence graph of the Fano plane, precisely because the lines of the Fano plane form a Steiner triple system.

\begin{lemma}\label{lm:condition_on_vectors}
Suppose $u_1,u_2,\ldots, u_7 \in \mathbb C^n$.  Then the following are equivalent.
\begin{enumerate}

	\item
	\label{cond:terminal_vector_condition}
	The vectors $u_1,u_2,\ldots,u_7$ occur as the $u_i$ of \eqref{eq:heawood_matrix_with_terminal_u_vectors} in some matrix $A \in M_{7+n,7}(\mathbb C)$ of that form having mutually orthogonal columns.

	\item \label{cond:inner_product_condition}
	The product $\langle u_i, u_j \rangle\langle u_j, u_k \rangle\langle u_k, u_i\rangle$ is real and negative
whenever $\{i,j,k\}$ is a line in the Fano plane, i.e., whenever $\{i,j,k\}$ is contained in the set \eqref{eqn:fano_plane_lines}.

\end{enumerate}
Moreover, if condition \eqref{cond:inner_product_condition} is satisfied with each $u_i \in \mathbb R^n$, then a matrix $A$ witnessing condition \eqref{cond:terminal_vector_condition} exists with $A \in M_{7+n,7}(\mathbb R)$.

\end{lemma}

\begin{proof}
Suppose first that condition \eqref{cond:terminal_vector_condition} is satisfied.  Then there exists some $A \in M_{7+n,7}(\mathbb C)$ with mutually orthogonal columns such that
\begin{equation*}
A = {\left[ \begin{array}{lllllllll}
  a & b & 0 & 0 & 0 & c & 0 \\
  0 & * & * & 0 & 0 & 0 & * \\
  * & 0 & * & * & 0 & 0 & 0 \\
  0 & * & 0 & * & * & 0 & 0 \\
  0 & 0 & * &  0& * & * & 0 \\
  0 & 0 & 0 & * & 0 & * & * \\
  * & 0 & 0 & 0 & * & 0 & * \\ \hline
  u_1 & u_2 & u_3 & u_4 & u_5 & u_6 & u_7
\end{array} \right]},
\end{equation*}
where $a$, $b$, $c$ and each $\ast$ entry are nonzero.
Since the columns indexed by the set $\{1,2,6\}$ are mutually orthogonal, it follows from Observation \ref{obs:combinatorics_of_fano_plane_upper_7_by_7} that
\begin{equation}\label{eqn:inner_product_equations}
a = \frac {-\langle u_1,u_2\rangle}{\overline b},
\quad
b = \frac {-\langle u_2, u_6 \rangle} {\overline c}
\quad
\text{ and }
\quad
c = \frac {-\langle u_6, u_1 \rangle} {\overline a}.
\end{equation}
Combining the first and third of these equations yields
\begin{equation}\label{eqn:expression_for_c}
c = \frac {-\langle u_6, u_1 \rangle} {\overline a} = -\langle u_6, u_1 \rangle \frac b {-\langle u_2,u_1\rangle} = \frac {\langle u_6, u_1 \rangle} {\langle u_2,u_1\rangle} b,\end{equation}
and combining this with the second equation of
\eqref{eqn:inner_product_equations} gives
\[ b = \frac {-\langle u_2, u_6 \rangle} {\overline c} = \frac {-\langle u_2, u_6 \rangle}{\overline b} \frac {\langle u_1, u_2 \rangle} {\langle u_1, u_6\rangle}, \]
which implies that 
\begin{equation}\label{eqn:equation_in_terms_of_b}
0 > -|b|^2 = \frac {\langle u_1, u_2 \rangle\langle u_2, u_6 \rangle} {\langle u_1, u_6\rangle}
\frac{\langle u_6, u_1\rangle}{\langle u_6, u_1\rangle} = \frac {\langle u_1, u_2 \rangle\langle u_2, u_6 \rangle\langle u_6, u_1\rangle } {|\langle u_1, u_6\rangle|^2}.
\end{equation}
In particular, then, $\langle u_1, u_2 \rangle\langle u_2, u_6 \rangle\langle u_6, u_1\rangle$ is real and negative. 
This same argument may be applied to every row of $A$, and hence condition \eqref{cond:inner_product_condition} holds.

Conversely, suppose condition \eqref{cond:inner_product_condition} holds.
Equations analogous to \eqref{eqn:inner_product_equations}, \eqref{eqn:expression_for_c} and \eqref{eqn:equation_in_terms_of_b} then yield values for the nonzero entries in each of the initial $7$ rows of a matrix $A$ of the form \eqref{eq:heawood_matrix_with_terminal_u_vectors}.  Explicitly, for each $m \in \{1,2,\ldots,7\}$, if the three nonzero entries in row $m$ fall in columns $i$, $j$ and $k$ with $i < j < k$, then values for those entries may be taken as
\begin{equation}\label{eqn:entries_from_inner_products}
A_{mj} =
 \sqrt{-\frac {\langle u_i, u_j \rangle\langle u_j, u_k \rangle\langle u_k, u_i\rangle } {|\langle u_i, u_k\rangle|^2}},
\quad
A_{mi} = -\frac{\langle u_i,u_j\rangle}{A_{mj}},
\quad \text{and}\quad
A_{mk} = A_{mj}\frac{\langle u_k, u_i \rangle}{\langle u_j,u_i\rangle}.
\end{equation}
It then follows from Observation \ref{obs:combinatorics_of_fano_plane_upper_7_by_7} that $A$ has mutually orthogonal columns. Hence, condition \eqref{cond:terminal_vector_condition} holds.
\end{proof}

Lemmas \ref{lm:condition_on_terminal_Us} and \ref{lm:condition_on_vectors} together show that the minimum semidefinite rank of the Heawood graph is the smallest value of $7+n$ such that vectors $u_1,u_2,\ldots,u_7 \in \mathbb{C}^n$ exist satisfying condition \eqref{cond:inner_product_condition} of Lemma \ref{lm:condition_on_vectors}.  In particular, to establish an upper bound of $10$ on
this value, it suffices to construct vectors in $\mathbb C^3$ satisfying this condition; we next show that in fact such vectors can be constructed in $\mathbb R^3$.

\begin{lemma}\label{lm:heptagonal_construction}
There exist vectors $u_1,u_2,\ldots,u_7 \in \mathbb R^3$ 
such that $\langle u_i, u_j \rangle\langle u_j, u_k \rangle\langle u_k, u_i\rangle$ is real and negative whenever $\{i,j,k\}$ is a line in the Fano plane, i.e., whenever $\{i,j,k\}$ is contained in the set \eqref{eqn:fano_plane_lines}.
\end{lemma}

\begin{proof}
Given a positive real number $\alpha$, let
\[ u_j = \left(\cos(2\pi j/7),\sin(2\pi j/7), \sqrt\alpha\, \right) \]
for each $j\in\{1,2,\ldots,7\}$.
Then, for any $j$ and $k$, 
\begin{eqnarray*}
\langle u_j,u_k \rangle &=& \cos(2\pi j/7)\cos(2\pi k/7) + \sin(2\pi j/7)\sin(2\pi k/7) + \alpha \\
&=& \cos( 2\pi(j-k)/7) + \alpha,
\end{eqnarray*}
so that $\langle u_j,u_k\rangle$ is completely determined by the difference $j-k$ modulo $7$.
But the set of pairwise differences modulo $7$ is the same for every set contained in \eqref{eqn:fano_plane_lines}; explicitly, it is $\{1,4,5\}$.  Hence, it suffices to ensure that the conclusion holds for any one such set, e.g., to guarantee that
$\langle u_1, u_2 \rangle\langle u_2, u_6 \rangle\langle u_6, u_1\rangle$
is real and negative.  This can be achieved by choosing $\alpha$ such that $\cos(3\pi/7) < \alpha  < \cos(\pi/7)$, as then
\begin{alignat*}{7}
\langle u_1,u_2 \rangle &=& \cos(2\pi/7) + \alpha &>& 0,\phantom{\cos(2\pi/7) + \alpha} \\
\langle u_2,u_6 \rangle &=& \cos(8\pi/7) + \alpha &=& -\cos(\pi/7) + \alpha & < 0, & \text{ and}\\
\langle u_6,u_1 \rangle &=& ~\cos(10\pi/7)+ \alpha &=& ~-\cos(3\pi/7)+\alpha  & > 0. &
 \qedhere
\end{alignat*}
\end{proof}

This yields an upper bound on the minimum semidefinite rank of the Heawood graph.

\begin{proposition}\label{prop:real_upper_bound}
The minimum semidefinite rank of the Heawood graph is at most $10$.
\end{proposition}

\begin{proof}
By Lemma \ref{lm:heptagonal_construction}, there exist $u_1,u_2,\ldots,u_7 \in \mathbb R^3$ satisfying condition \eqref{cond:inner_product_condition} of Lemma \ref{lm:condition_on_vectors}, and hence appearing as the $u_i$ of \eqref{eq:heawood_matrix_with_terminal_u_vectors} for some matrix $A \in M_{10,7}(\mathbb R)$ with orthogonal columns.  Hence, by Lemma \ref{lm:condition_on_terminal_Us}, the Heawood graph has an orthogonal representation in $\mathbb R^{10}$, and hence in $\mathbb{C}^{10}$.
\end{proof}

To establish our main result, we turn now to the requisite lower bound, namely that the minimum semidefinite rank of the Heawood graph is at least $10$.  By the discussion above, it suffices to show that no vectors from $\mathbb C^2$ exist satisfying the conditions of Lemma \ref{lm:condition_on_vectors}.  Our approach can be summarized as follows.  Assuming to the contrary that such vectors do exist, we identify them with points on the Riemann sphere.  We then argue that the two conditions on these vectors shown to be equivalent by Lemma \ref{lm:condition_on_vectors} are themselves equivalent to a condition on the corresponding points on the sphere that, when satisfied, implies that these points must be arranged in such a way as to induce a $2$-coloring of the Fano plane with no monochromatic line, in contradiction to Lemma \ref{lem:fano_plane_not_2_colorable}.

Our first task is to establish the appropriate correspondence between vectors in $\mathbb C^2$ and points on the appropriate sphere in $\mathbb R^3$.  We begin with a crucial observation.

\begin{observation}\label{obs:scaling_invariance}
Each of the two conditions of Lemma \ref{lm:condition_on_vectors} is unaffected by multiplying any individual vector by an arbitrary nonzero complex scalar.
\end{observation}

In light of Observation \ref{obs:scaling_invariance}, we may regard the conditions of Lemma \ref{lm:condition_on_vectors} as applying to points on the projective line $\mathbb CP^1$, which can be thought of as the extended complex plane, $\mathbb C \cup \{\infty\}$.  Through the usual stereographic projection, the extended complex plane can be transformed bijectively to a sphere in $\mathbb R^3$.  The image of such an identification is typically referred to as the \textit{Riemann sphere}. (See \cite[Section 3.IV]{needham} for details.)  Any sphere in $\mathbb R^3$ can be made the image of such an identification; for the sake of making our computations explicit in what follows, we will choose the sphere of radius $1/2$ centered at $(0,0,1/2)$, namely
\begin{equation}\label{eq:def_of_sphere}
S = \{ (x,y,z) : x^2+y^2+(z-1/2)^2 = (1/2)^2 \} = \{ (x,y,z) : x^2+y^2+z^2 = z \}.
\end{equation}
Again for the sake of explicit computation, we now define a function $\varphi$ that effects the identification outlined above, mapping points in $\mathbb C^2$ to points on $S$.

\begin{definition}\label{def:def_of_phi_map}
Let $\varphi:\mathbb C^2 \rightarrow \mathbb R^3$ be defined as follows.
First, let $\varphi_1$ map $\mathbb C^2$ to $\mathbb CP^1$ in the usual way, i.e.,
\[ \varphi_1(z_1,z_2) = \begin{cases} z_1/z_2 & \text{if } z_2\not= 0, \\ \infty & \text{otherwise}. \end{cases} \]
Next, apply the familiar stereographic projection to map the image of $\varphi_1$ to the sphere $S$ defined in \eqref{eq:def_of_sphere}.  Specifically, identify $\infty$ with the ``pole'' of the sphere at $(0,0,1)$, and identify $a+bi \in \mathbb C$ with the unique point at which $S\setminus\{(0,0,1)\}$ intersects the line parameterized by
\[ t(a,b,0) + (1-t)(0,0,1),\quad t \in \mathbb R. \]
It follows from \eqref{eq:def_of_sphere} that this point of intersection is $\frac 1{1+a^2+b^2}(a,b,a^2+b^2)$.  Thus, we let
\[ \varphi_2(z) = \begin{cases}
	\frac 1{1+a^2+b^2}(a,b,a^2+b^2) & \text{if } z=a+bi, \text{ and} \\
	(0,0,1) & \text{if } z = \infty.
\end{cases} \]
Finally, let $\varphi = \varphi_2 \circ \varphi_1$.
\end{definition}

Having identified each vector in $\mathbb C^2$ with a point on the sphere $S$,
 the conditions of Lemma \ref{lm:condition_on_vectors} applied to triples of vectors in $\mathbb C^2$ can be reinterpreted as applying to triples of points on $S$.  The next lemma provides two equivalent such interpretations.

\begin{lemma}\label{lm:condition_on_points_of_S2}
Let $C$ denote the center of the sphere $S$ defined in \eqref{eq:def_of_sphere}, i.e., $C=(0,0,1/2)$.  For any $u_i$, $u_j, u_k \in \mathbb C^2$, the following are equivalent.
\begin{enumerate}
	\item\label{cond:ip_cond} The product $\langle u_i,u_j \rangle \langle u_j,u_k \rangle \langle u_k,u_i \rangle$ is real and negative.
	\item\label{cond:center_in_convex_hull} No two of $\varphi(u_i)$, $\varphi(u_j)$ and $\varphi(u_k)$ are antipodal on $S$, but the convex hull of all three contains $C$.
	\item\label{cond:separating_plane_condition} Every plane passing through $C$ that contains none of $\varphi(u_i)$, $\varphi(u_j)$ and $\varphi(u_k)$ separates one of those latter three points from the other two.  Moreover, each of those three points is separated from the other two by some such plane.
\end{enumerate}
\end{lemma}

\begin{proof}
We start with some simplifying assumptions.  First, subject to the appropriate scaling, we may assume that each of $u_i$, $u_j$ and $u_k$ is equal either to $(1,0)$ or to $(z,1)$ for some $z\in\mathbb C$.  (That is, we may work projectively.)  Since it is clear from Definition \ref{def:def_of_phi_map} that the image of a point under $\varphi$ is  determined only by the line through the origin in $\mathbb C^2$ on which the point lies, this cannot affect conditions \eqref{cond:center_in_convex_hull} or \eqref{cond:separating_plane_condition}, while by Observation \ref{obs:scaling_invariance} it does not affect condition \eqref{cond:ip_cond}.

Next, observe that some pair of rotations of the sphere can be applied sequentially to move $\varphi(u_i)$ to the origin and $\varphi(u_j)$ to a point on the $xz$-plane with a nonnegative $x$-coordinate.  This clearly leaves conditions  \eqref{cond:center_in_convex_hull} and \eqref{cond:separating_plane_condition} unaffected.  Moreover, such a rigid motion of the sphere corresponds to a unitary transformation of $\mathbb C^2$
\cite[Section 6.II]{needham}
and hence preserves condition \eqref{cond:ip_cond}.  Hence, we may assume that $\varphi(u_i)=(0,0,0)$, so that, equivalently, $u_i = (0,1)$,
and also that for some nonnegative real number $s$,
\[ \varphi(u_j)={\textstyle \frac 1{1+s^2}}(s,0,s^2), \text{ so that, equivalently, } u_j = (s,1). \]
These assumptions are illustrated in Figure \ref{fig:sphere_with_points}.

\begin{figure}[h]

\tikzset{%
  >=latex,
  inner sep=0pt,%
  outer sep=2pt,%
  mark coordinate/.style={inner sep=0pt,outer sep=0pt,minimum size=3pt,
    fill=black,circle}%
}

\newcommand\pgfmathsinandcos[3]{%
  \pgfmathsetmacro#1{sin(#3)}%
  \pgfmathsetmacro#2{cos(#3)}%
}
\newcommand\LongitudePlane[3][current plane]{%
  \pgfmathsinandcos\sinEl\cosEl{#2} 
  \pgfmathsinandcos\sint\cost{#3} 
  \tikzset{#1/.estyle={cm={\cost,\sint*\sinEl,0,\cosEl,(0,0)}}}
}
\newcommand\LatitudePlane[3][current plane]{%
  \pgfmathsinandcos\sinEl\cosEl{#2} 
  \pgfmathsinandcos\sint\cost{#3} 
  \pgfmathsetmacro\yshift{\cosEl*\sint}
  \tikzset{#1/.estyle={cm={\cost,0,0,\cost*\sinEl,(0,\yshift)}}} %
}
\newcommand\DrawLongitudeCircle[2][1]{
  \LongitudePlane{\angEl}{#2}
  \tikzset{current plane/.prefix style={scale=#1}}
  \pgfmathsetmacro\angVis{atan(sin(#2)*cos(\angEl)/sin(\angEl))} %
  \draw[current plane] (\angVis:1) arc (\angVis:\angVis+180:1);
  \draw[current plane,dashed] (\angVis-180:1) arc (\angVis-180:\angVis:1);
}
\newcommand\DrawLatitudeCircle[2][1]{
  \LatitudePlane{\angEl}{#2}
  \tikzset{current plane/.prefix style={scale=#1}}
  \pgfmathsetmacro\sinVis{sin(#2)/cos(#2)*sin(\angEl)/cos(\angEl)}
  \pgfmathsetmacro\angVis{asin(min(1,max(\sinVis,-1)))}
  \draw[current plane] (\angVis:1) arc (\angVis:-\angVis-180:1);
  \draw[current plane,dashed] (180-\angVis:1) arc (180-\angVis:\angVis:1);
}
{
\begin{tikzpicture}

\def\R{2.2}
\def\Rminusepsilon{2.19}
\def\angEl{40}
\def\angAz{-130}

\pgfmathsetmacro\H{\R*cos(\angEl)}
\tikzset{xyplane/.estyle={cm={cos(\angAz),sin(\angAz)*sin(\angEl),-sin(\angAz),
                              cos(\angAz)*sin(\angEl),(0,-\H)}}}
\LongitudePlane[xzplane]{\angEl}{\angAz}
\LatitudePlane[equator]{\angEl}{0}

\draw (0,0) circle (\R);
\fill[ball color=white] (0,0) circle (\Rminusepsilon);

\coordinate (O) at (0,0);
\coordinate[mark coordinate] (C) at (0,0);
\coordinate[mark coordinate] (S) at (0,-\H);

\path[xzplane] (0.65*\R,1.5) coordinate[mark coordinate] (XE);

\DrawLongitudeCircle[\R]{\angAz}

\draw[xyplane,->] (-1*\R,0) -- (01*\R,0) node[below] {$x$};
\draw[xyplane,->] (0,-1*\R) -- (0,1*\R) node[below] {$y$};
\draw[->] (0,-\H) -- (0,1.15*\R) node[above] {$z$};

\draw[xzplane,dashed] (0,1.5) -- (0.65*\R,1.5);
\draw[xzplane,dashed] (0.65*\R,1.5) -- (0.65*\R,-1*\R);

\path (C) +(1.4ex,0.1ex) node[below] {$C$};
\path (S) +(4.1ex,2.1ex) node[below] {$\varphi(u_i)$};
\path (XE) +(1.25ex,0.5ex) node[above left] {$\varphi(u_j)$};

\end{tikzpicture}
}

\caption{Illustration of assumptions on the placement of the points $\varphi(u_i)$ and $\varphi(u_j)$ on the sphere $S$ in the proof of Lemma \ref{lm:condition_on_points_of_S2}.}\label{fig:sphere_with_points}
\end{figure}
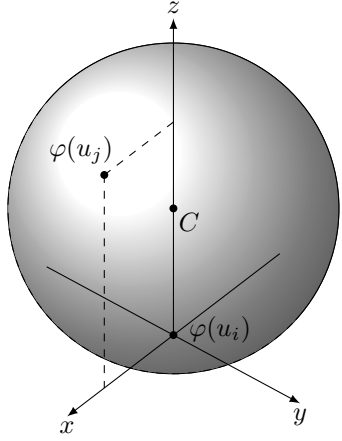

Finally, for any $z \in \mathbb C^2$, let $\overline \varphi(z)$ denote the point on $S$ antipodal to $\varphi(z)$.
In particular,
\begin{equation}\label{eqn:phi_of_uj_expression}
\overline\varphi(u_j)= {\textstyle\frac 1{1+s^2}} \left( -s,0, 1 \right).
\end{equation}
We now begin the proof by showing conditions \eqref{cond:ip_cond} and \eqref{cond:center_in_convex_hull} to be equivalent.  Suppose first that \eqref{cond:ip_cond} holds.  This is incompatible with $u_k=(1,0)$, since $u_i=(0,1)$.  Therefore $u_k=(t,1)$ for some $t\in\mathbb C$.  Hence, $\langle u_i,u_j \rangle \langle u_j,u_k \rangle \langle u_k,u_i \rangle  = 1+st$ is real and negative by \eqref{cond:ip_cond}, and so $t \in \mathbb R$ with $t < 0$ and $s > 0$.
Thus,
\begin{equation}\label{eqn:phi_of_uk_expression}
\varphi(u_k)={\textstyle\frac 1{1+t^2}}(t,0,t^2),
\end{equation}
and so $\varphi(u_k)$ lies on the $xz$-plane with a negative $x$-coordinate.
Moreover,
\[ 1+st < 0 \implies |st| > 1 \implies s^2t^2 > 1 \implies 1+t^2 < t^2(1+s^2) \implies {\textstyle\frac 1{1+s^2}} < {\textstyle\frac {t^2}{1+t^2}}. \]
By \eqref{eqn:phi_of_uj_expression} and \eqref{eqn:phi_of_uk_expression}, this shows that the $z$-coordinate of $\varphi(u_k)$ exceeds that of  $\overline\varphi(u_j)$, so that $C$ is in the convex hull of $\varphi(u_i)$, $\varphi(u_j)$ and $\varphi(u_k)$, and also shows that $\varphi(u_j) \not= \overline\varphi(u_k)$.  Moreover, neither $\varphi(u_j)$ nor $\varphi(u_k)$ may equal $\overline\varphi(u_i)=(0,0,1)$. Hence, the three points $\varphi(u_i)$, $\varphi(u_j)$ and $\varphi(u_k)$ do not contain an antipodal pair.  Thus, condition \eqref{cond:center_in_convex_hull} holds.

Now suppose that \eqref{cond:center_in_convex_hull} holds.  Then $s>0$, as $s=0$ would give $\varphi(u_j)=\varphi(u_i)$, requiring $\varphi(u_k) = \overline\varphi(u_i)$ in order that $C$ lie in the convex hull of the three points.  Further, since $\varphi(u_k) \not= \overline\varphi(u_i)$, we cannot have $u_k=(1,0)$.  Therefore, $u_k=(t,1)$ for some $t \in \mathbb C$.
The fact that $C$ is in the convex hull of $\varphi(u_i)$, $\varphi(u_j)$ and $\varphi(u_k)$ implies that $\varphi(u_k)$ lies on the plane containing $\varphi(u_i)$, $\varphi(u_j)$ and $C$, namely the $xz$-plane.  Thus, we have $t\in \mathbb R$, so that
\[ \varphi(u_k) = {\textstyle\frac 1{1+t^2}}(t,0,t^2).\]  Moreover,
since $\varphi(u_k)$ and $\varphi(u_j)$
must lie on opposite sides of
the $yz$-plane,
we have $t<0$.
Finally, the $z$-coordinate of $\varphi(u_k)$ must exceed that of $\overline\varphi(u_j)$, so that
$\frac{t^2}{1+t^2} > \frac 1{1+s^2}$.
This gives $t^2+t^2s^2 = t^2(1+s^2) > 1+t^2$.  Hence, $s^2t^2 > 1$, so that $|st| > 1$.  Combining this with the fact that $t<0$ and $s > 0$ so that $st < 0$, we have
$\langle u_i,u_j \rangle \langle u_j,u_k \rangle \langle u_k,u_i \rangle=1+st < 0$, so that condition \eqref{cond:ip_cond} holds.

Having shown that conditions \eqref{cond:ip_cond} and \eqref{cond:center_in_convex_hull} are equivalent, we now complete the proof by proving the equivalence of conditions \eqref{cond:center_in_convex_hull} and \eqref{cond:separating_plane_condition}.
Assume first that \eqref{cond:center_in_convex_hull} holds.  Then $C$ is in the convex hull of $\varphi(u_i)$, $\varphi(u_j)$ and $\varphi(u_k)$, so that all three points lie on the $xz$-plane.  Now consider a plane $P$ passing through $C$ that contains none of $\varphi(u_i)$, $\varphi(u_j)$ and $\varphi(u_k)$.  As $P$ may not coincide with the $xz$-plane, its intersection with the $xz$-plane is a line $L$ passing through $C$.  Since $C$ is in the convex hull of $\varphi(u_i)$, $\varphi(u_j)$ and $\varphi(u_k)$, these three points cannot lie all on the same side of $L$.  This implies that $L$, and hence $P$, separates one of the three points from the other two.

It remains to show that each of $\varphi(u_i)$, $\varphi(u_j)$ and $\varphi(u_k)$ is separated from the other two by some plane containing $C$.  By symmetry, it suffices to prove that $\varphi(u_i)$ is separated from $\varphi(u_j)$ and $\varphi(u_k)$ by some such plane.  Toward that end, consider the plane $P$ perpendicular to the $xz$-plane and passing through $\varphi(u_j)$ and $\overline\varphi (u_j)$.  Since $C$ is in the convex hull of the three points, $\varphi(u_k)$ cannot lie on the same side of $P$ as does $\varphi(u_i)$.  Moreover, $\varphi(u_k)$ cannot lie on the plane $P$, as this would imply $\varphi(u_k)=\overline\varphi(u_j)$, which \eqref{cond:center_in_convex_hull} forbids.  Hence, $\varphi(u_i)$ and $\varphi(u_k)$ must lie on opposite sides of $P$.  Since $P$ contains the line through $C$ that is perpendicular to the $xz$-plane, it follows that rotating $P$ about this line by  a sufficiently small angle produces a plane through $C$ separating $\varphi(u_i)$ from $\varphi(u_j)$ and $\varphi(u_k)$, as desired.

Now suppose \eqref{cond:separating_plane_condition} holds. If any two points among $\varphi(u_i)$, $\varphi(u_j)$ and $\varphi(u_k)$ were antipodal, then those two points could not be separated from the third by any plane containing $C$, which would contradict \eqref{cond:separating_plane_condition}.  Hence, $\varphi(u_i)$, $\varphi(u_j)$ and $\varphi(u_k)$ do not contain an antipodal pair, and it remains to show that $C$ is in their convex hull.

First, since $\varphi(u_i)$ and $\varphi(u_j)$ are not antipodal, they lie on the same side of some line in the $xz$-plane passing through $C$.  If $\varphi(u_k)$ were not in the $xz$-plane, then rotating the $xz$-plane about this line by some small angle would produce a plane relative to which all three of $\varphi(u_i)$, $\varphi(u_j)$ and $\varphi(u_k)$ would lie on the same side, contradicting \eqref{cond:separating_plane_condition}.  Hence, $\varphi(u_k)$ lies on the $xz$-plane along with $\varphi(u_i)$, $\varphi(u_j)$ and $C$.

We have by \eqref{cond:separating_plane_condition} that $\varphi(u_i)$ is separated from $\varphi(u_j)$ and $\varphi(u_k)$ by some plane $P$ that contains none of those points but does contain $C$.  As $P$ may not coincide with the $xz$-plane, it intersects the $xz$-plane in some line $L$ passing through $C$.  Let $A$ be the point at which the line through $\varphi(u_i)$ and $\varphi(u_j)$ intersects $L$ and let $B$ be the point at which the line through $\varphi(u_i)$ and $\varphi(u_k)$ intersects $L$.  (See Figure \ref{fig:circle_diagram}.)

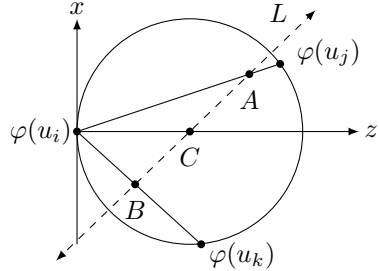
\begin{figure}[h]
\tikzset{%
  >=latex, %
  inner sep=0pt,%
  outer sep=2pt,%
  mark coordinate/.style={inner sep=0pt,outer sep=0pt,minimum size=3pt,
    fill=black,circle}%
}

\begin{tikzpicture}
\def\R{1.5}
\def\theta{44}

\node[draw,minimum size=2cm*\R,inner sep=0,outer sep=0,circle] (C) at (1.0*\R,0) {};
\coordinate[mark coordinate] (O) at (20:0*\R);
\draw (O) node[left] {$\varphi(u_i)$};

\coordinate[mark coordinate] (C) at (1*\R,0);
\draw (C) node[below=1ex] {$C$};

\draw[<-] (0,1*\R) node[above] {$x$} -- (0,-1*\R);
\draw[->] (O) -- (2.5*\R,0) node[right] {$z$};

\coordinate[mark coordinate] (phi_uj) at (1.8*\R,0.6*\R);
\draw (phi_uj) node[above=1ex,right=1ex] {$\varphi(u_j)$};

\coordinate[mark coordinate] (phi_uk) at (1.1*\R,-1*\R);
\draw (phi_uk) node[below=1ex,right=0.25ex] {$\varphi(u_k)$};

\path (C) ++(\theta:1.55*\R) coordinate (L_end1);
\path (C) ++(\theta+180:1.65*\R) coordinate (L_end2);
\draw[dashed,->] (C) -- (L_end1) node[left=2ex] {$L\vphantom{^|}$};
\draw[dashed,->] (C) -- (L_end2);

\draw (O) -- (phi_uj);
\draw (O) -- (phi_uk);

\coordinate[mark coordinate] (A) at (intersection of O--phi_uj and C--L_end1);
\draw (A) node[below=1ex] {$A$};

\coordinate[mark coordinate] (B) at (intersection of O--phi_uk and C--L_end1);
\draw (B) node[below=1ex] {$B$};

\end{tikzpicture}

\caption{Representative arrangement of the points $\varphi(u_i)$, $\varphi(u_j)$ and $\varphi(u_k)$ as in the argument that condition \eqref{cond:separating_plane_condition} implies condition \eqref{cond:center_in_convex_hull} in the proof  of Lemma \ref{lm:condition_on_points_of_S2}.  The circle is the intersection of the $xz$-plane with the sphere $S$.}\label{fig:circle_diagram}
\end{figure}

Note that $\varphi(u_j)$ and $\varphi(u_k)$ must lie on opposite sides of the $yz$-plane, as otherwise rotating that plane by some small angle about the line through $C$ that is perpendicular to the $zx$-plane would produce a plane containing $C$ on one side of which would lie all three of the points $\varphi(u_i)$, $\varphi(u_j)$ and $\varphi(u_k)$, contradicting \eqref{cond:separating_plane_condition}.  It follows that $A$ and $B$ lie on opposite sides of the $yz$-plane as well.  This implies that $C$ is on the line segment with endpoints $A$ and $B$.  Since $A$ and $B$ were chosen within the convex hull of $\varphi(u_i)$, $\varphi(u_j)$ and $\varphi(u_k)$, it follows  that $C$ lies in the convex hull of those points as well.  Hence, condition \eqref{cond:center_in_convex_hull} holds.
\end{proof}

We now have that any collection of vectors in $\mathbb C^2$ satisfying the algebraic condition \eqref{cond:ip_cond} of Lemma \ref{lm:condition_on_vectors} corresponds to a collection of points on the sphere $S$ arranged such that every triple of points corresponding to a line in the Fano plane satisfies the geometric conditions \eqref{cond:center_in_convex_hull} and \eqref{cond:separating_plane_condition} of Lemma \ref{lm:condition_on_points_of_S2}.  We next show that such an arrangement gives rise to an impossible coloring of the Fano plane, a contradiction that yields our desired lower bound.

\begin{proposition}\label{prop:complex_lower_bound}
The minimum semidefinite rank of the Heawood graph is at least $10$.
\end{proposition}

\begin{proof}
Suppose to the contrary that the Heawood graph has an orthogonal representation in $\mathbb C^9$.  Then, by Lemma \ref{lm:condition_on_terminal_Us}, there exist vectors $u_1,u_2,\ldots,u_7 \in \mathbb C^2$ satisfying the equivalent conditions of Lemma \ref{lm:condition_on_vectors}.  By Lemma \ref{lm:condition_on_points_of_S2}, these vectors induce points $\varphi(u_1),\varphi(u_2),\ldots,\varphi(u_7)$ on the sphere $S$ defined in \eqref{eq:def_of_sphere} such that whenever $\{i,j,k\}$ is a line in the Fano plane, i.e., whenever $\{i,j,k\}$ is contained in the set \eqref{eqn:fano_plane_lines}, the triple of points $\varphi(u_i)$, $\varphi(u_j)$ and $\varphi(u_k)$ satisfies condition \eqref{cond:separating_plane_condition} of Lemma \ref{lm:condition_on_points_of_S2}.

Let $P$ be any plane through the center of $S$ that contains none of the points \linebreak $\varphi(u_1),\varphi(u_2),\ldots,\varphi(u_7)$.  Then $P$ divides $S$ into hemispheres.  Choose one of the hemispheres, and color red every point $i$ of the Fano plane such that $\varphi(u_i)$ lies on that hemisphere.  Color the other points of the Fano plane green.  By Corollary \ref{lem:fano_plane_not_2_colorable}, this coloring must result in some line of the Fano plane, say $\{r,s,t\}$, all of whose points are colored the same.  But this means that $\varphi(u_r)$, $\varphi(u_s)$ and $\varphi(u_t)$ lie all on the same hemisphere of $S$, contradicting condition \eqref{cond:separating_plane_condition} of Lemma \ref{lm:condition_on_points_of_S2}.
\end{proof}

Our main result now follows from the combination of Propositions \ref{prop:real_upper_bound} and \ref{prop:complex_lower_bound}.

\begin{theorem}
The minimum semidefinite rank of the Heawood graph is $10$.
\end{theorem}

\section{Incidence graphs of Steiner triple systems}
\label{sec:generalization}

We now identify the Heawood graph as one of a general family of graphs to which the approach of Section \ref{sec:heawood_msr} may be applied.  Recall the following definition from combinatorial design theory; see, e.g., \cite[Chapter 2]{comb_design_handbook}.

\begin{definition}\label{def:steiner_system}
A \textit{Steiner triple system} of order $v$ consists of a set $X$, whose elements are called the \textit{points} of the system, such that $|X|=v$, together with a collection of $3$-subsets of $X$, called the \textit{triples} of the system, such that every $2$-subset of $X$ is contained in exactly one triple.
\end{definition}

It follows from Observation \ref{obs:lines_of_fano_plane} that the lines of the Fano plane form a Steiner triple system of order $7$.  (Actually it is the unique Steiner triple system of that order.)  A fact crucial to the proof of Proposition \ref{prop:complex_lower_bound} was previously noted as Lemma \ref{lem:fano_plane_not_2_colorable}, namely that every $2$-coloring of the points of the Fano plane
induces a monochromatic line.  More generally, the \textit{weak chromatic number} of a Steiner triple system is the smallest number of colors from which the points of the system may be colored such that no triple is left with all of its points colored the same; Lemma \ref{lem:fano_plane_not_2_colorable} is a special case of the following result of \cite{rosa}.

\begin{theorem}[\cite{rosa}]\label{thm:steiner_3_colorable}
	Every Steiner triple system of order $7$ or greater has a weak chromatic number of at least $3$.
\end{theorem}

Every Steiner triple system is represented by a bipartite graph in the same sense in which the Fano plane is represented by the Heawood graph.

\begin{definition}
The \textit{incidence graph} of a Steiner triple system is the graph $G$ whose vertices can be partitioned into two sets, one in correspondence with the points of the system and the other in correspondence with its triples, such that two vertices are adjacent precisely when they correspond to a point and a triple containing that point.
\end{definition}

Hence, the Heawood graph is the incidence graph of the unique Steiner triple system of order $7$.  The approach developed in Section \ref{sec:heawood_msr} to establish a lower bound on the minimum semidefinite rank of the Heawood graph can be adapted to do the same for the incidence graph of any Steiner triple system of order at least $7$.

\begin{theorem}\label{thm:general_steiner_triple_system_analog}
Let $G$ be the incidence graph of a Steiner triple system of order $v \ge 7$, let $b$ be the number of triples of the system, and let $n$ be the number of vertices of $G$.  Then $b = { \textstyle \frac 13{v \choose 2} }$, $n= b + v = \textstyle \frac 16 (v^2+5v)$, and
\begin{equation*}\label{eq:steiner_system_general_bound}
\msr{G} \geq\, b + 3 = \textstyle\frac 16 (v^2-v+18).
\end{equation*}
\end{theorem}

\begin{proof}[Proof sketch]
That $b = { \textstyle \frac 13{v \choose 2} }$ follows immediately from Definition \ref{def:steiner_system}.  The claim that $n=b+v$ is trivial. By definition, $G$ has a biadjacency matrix $M$ of size $b \times v$.
With the zero-nonzero pattern of $M$ playing the role of the upper portion of \eqref{eq:heawood_matrix_with_terminal_u_vectors}, a result analogous to Lemma \ref{lm:condition_on_terminal_Us} is obtained by the same argument.
It follows from Definition \ref{def:steiner_system} that, for any matrix whose initial $b$ rows have a zero-nonzero pattern matching that of $M$, a statement analogous to Observation \ref{obs:combinatorics_of_fano_plane_upper_7_by_7} holds.  Hence, the statement and proof of Lemma \ref{lm:condition_on_vectors} can be adapted in a straightforward way, and Lemma \ref{lm:condition_on_points_of_S2} can then be applied without modification.

The conclusion of the argument then proceeds as in the proof of Proposition \ref{prop:complex_lower_bound}.  That is, any supposed orthogonal representation of $G$ in fewer than $b+3$ dimensions gives rise to $v$ points on the Riemann sphere arranged so as to induce a $2$-coloring of the points of the Steiner triple system in which at least two different colors occur within every triple, contradicting 
Theorem \ref{thm:steiner_3_colorable}.
\end{proof}

We wish to point out the limitations of Theorem \ref{thm:general_steiner_triple_system_analog},  so as to make clear why we did not attempt to develop our main results in such general terms.  To this end, note that the incidence graph $G$ of a Steiner triple system of order $v$ has an independent set (corresponding to the triples of the system) of size $b=\frac 13{v \choose 2}$.  This trivially implies that $\msr{G} \ge b = \frac 13{v \choose 2} = \frac 16(v^2 - v)$, and Theorem \ref{thm:general_steiner_triple_system_analog}  provides only a slight improvement on this bound.  In the case of the Heawood graph, what is interesting is that this improved bound is sharp.  It is unclear whether this remains the case for the incidence graphs of larger Steiner triple systems, however.  Nevertheless, there are many (see \cite[p.\ 15]{comb_design_handbook}) Steiner triple systems of small order, and the application of
Theorem \ref{thm:general_steiner_triple_system_analog} to their incidence graphs may be illuminating.

\section{Conclusion and open questions}\label{sec:conclusion}

Theorem \ref{thm:general_steiner_triple_system_analog} gives a lower bound on $\msr{G}$ whenever $G$ is the incidence graph of a Steiner triple system.  It is natural to compare this bound with that obtained from the positive semidefinite zero forcing number $\zplus{G}$ referenced in Section \ref{sec:intro}.  Table \ref{tab:msr_vs_zf_for_sts} details the result of this comparison for each Steiner triple system of order $v$, with $v > 3$ to avoid the trivial case, up to $v = 15$.
(The next order for which any Steiner triple systems exist is $v=19$, but in this case it would be computationally expensive to determine $\zplus{G}$ for even just one of these, and there are altogether $11{,}084{,}874{,}829$ of them \cite{kaski}.)

\begin{table}
\caption{\label{tab:msr_vs_zf_for_sts}
Comparison of the lower bound on $\msr{G}$ provided by Theorem  \ref{thm:general_steiner_triple_system_analog} with that implied by the positive semidefinite zero forcing number for the incidence graph $G$ of each Steiner triple system of each of the four smallest possible orders.
}

\begin{tabular}{C{0.475in}@{}C{0.475in}C{0.8in}C{0.62in}C{0.8in}C{0.8in}C{1.1in}} \toprule
	\multicolumn{2}{C{0.95in}}{Steiner triple system parameters} &
		Number of Steiner triple systems &
		Number of vertices in $G$ &
		Positive semidefinite zero forcing number &
		Bound from zero forcing &
		Bound from Theorem \ref{thm:general_steiner_triple_system_analog} \\
		
		\cmidrule(lr){1-2}\cmidrule(lr){4-4}\cmidrule(lr){5-5}\cmidrule(lr){6-6}\cmidrule(lr){7-7}

			$v$ & $b$ &
			&
			$n=v+b$ &
			$\zplus{G}$ &
			$n-\zplus{G}$ &
			$b+3 = \frac 13{v \choose 2} + 3$ \\
		\midrule

	7 & 7 & 1 & 14 & 5 & 9 & 10 \\ 
	9 &  12 & 1 & 21 & 7 & 14 & 15 \\ 
	13 & 26 & 2 & 39 & 11 & 28 & 29 \\ 
	15 & 35 & 80 & 50 & 13 & 37 & 38 \\
	\bottomrule
		
\end{tabular}
\end{table}

Table \ref{tab:msr_vs_zf_for_sts} invites some observations.  The first is that in each case the lower bound obtained from Theorem \ref{thm:general_steiner_triple_system_analog} exceeds the lower bound provided by the zero forcing number by exactly one.  This happens to be the case because, for each graph $G$ of the $84$ detailed in the table, $\zplus{G}$ turns out to be $2$ less than the order of the corresponding Steiner triple system.  The question as to whether this holds in general is outside the scope of the present work, but seems interesting.

\begin{question}\label{q:pattern_in_table_continues}
Does $\zplus{G} = v - 2$ whenever $G$ is the incidence graph of a Steiner triple system of order $v$?
\end{question}

An affirmative answer to Question \ref{q:pattern_in_table_continues} would imply that the positive semidefinite zero forcing number of the incidence graph of a Steiner triple system of order $v$ is determined by $v$ alone.  It is open as well whether this may be the case for the positive semidefinite minimum rank itself.

\begin{question}\label{q:noniso_STS_with_differing_msr}
Do there exist two nonisomorphic Steiner triple systems of the same order whose incidence graphs differ in their minimum semidefinite rank?
\end{question}

In particular, although the lower bound provided by Theorem \ref{thm:general_steiner_triple_system_analog} is met by the Heawood graph, we do not expect that this is uniformly the case for the incidence graphs of Steiner triple systems of larger order.  Nevertheless, the question remains open even for the unique Steiner triple system of order $9$.

\begin{question}
Is the Heawood graph the only incidence graph of a Steiner triple system for which the lower bound of Theorem \ref{thm:general_steiner_triple_system_analog} is met?  In particular, with $G$ the incidence graph of the unique Steiner triple system of order $9$,
does an orthogonal representation of $G$ in $\mathbb C^{15}$ exist?
\end{question}

Given a lower bound on the minimum semidefinite rank, the problem of establishing a corresponding upper bound is often
handled via some appropriate geometric construction.  Here this is done for the Heawood graph via Lemma \ref{lm:heptagonal_construction}.  For the incidence graphs of larger Steiner triple systems, however, the problem remains open.

\begin{question}
Can properties of Steiner triple systems in general be exploited to construct  low-dimensional orthogonal representations for their incidence graphs?
\end{question}

Of course, the questions explored here for Steiner triple systems may be considered for the incidence graphs of other families of combinatorial designs.  By definition, such graphs are bipartite, and so a natural analog of Lemma \ref{lm:condition_on_terminal_Us} is always available. No appropriate generalization of Lemma \ref{lm:condition_on_vectors}, however, seems forthcoming in any case beyond that of a Steiner triple system.

For the case of Steiner triple systems, Lemma \ref{lm:condition_on_points_of_S2} gives a useful geometric interpretation of the conditions on the vectors $u_i$ of Lemma \ref{lm:condition_on_vectors}, and this was crucial to the approach used to obtain the lower bound of Theorem \ref{thm:general_steiner_triple_system_analog}.
This raises the question as to whether there can be found some analogous geometric interpretation for these conditions as they apply to vectors in $\mathbb C^k$ for $k \ge 3$.  Such an interpretation might provide an avenue toward
generalizing the lower bound established here for the Heawood graph to the minimum semidefinite ranks of the incidence graphs of other Steiner triple systems.

\section{Acknowledgments}

The present work developed through a collaboration of the authors that began at the 2010 NSF-CBMS Regional Research Conference entitled \textit{The Mutually Beneficial Relationship of Matrices and Graphs}, supported by the IMA and by the NSF through grant number DMS-0938261.  The authors wish to thank those organizations as well as Iowa State University, which hosted the meeting.


\bibliographystyle{plain}
\bibliography{Orthrep-Steiner-arXiv}

\end{document}